\documentclass[reqno]{amsart}
\usepackage{amsaddr}
\usepackage{latexsym,upref,enumerate}
\usepackage{nicefrac}
\usepackage{hyperref}
\usepackage{mathtools}
\usepackage{cases}
\usepackage[final]{microtype} 
\usepackage{color}
\usepackage[pdftex,dvipsnames]{xcolor}  
\usepackage[colorinlistoftodos,prependcaption,textsize=tiny]{todonotes}
\usepackage{xargs}

\newcommand{\be}{\begin{equation}}
\newcommand{\ee}{\end{equation}}
\newcommand{\bew}{\begin{equation*}} 
\newcommand{\eew}{\end{equation*}}

\newcommand{\eps}{\varepsilon}

\newcommand{\R}{\mathbb R}
\newcommand{\T}{\mathbb T}
\newcommand{\Z}{\mathbb Z}
\newcommand{\p}{\mathrm{per}}
\newcommand{\e}{\mathrm e}
\newcommand{\I}{\mathrm{i}}
\newtheorem{theorem}{Theorem}[section]
\newtheorem{lemma}[theorem]{Lemma}

\theoremstyle{definition}

\newtheorem{remark}[theorem]{Remark}

\numberwithin{equation}{section}

\begin{document}

\title[Well-posedness of a highly nonlinear shallow water equation]{Well-posedness of a highly nonlinear\\shallow water equation on the circle}

\author[N.~Duruk Mutlubas]{Nilay Duruk Mutlubas}
\address{Faculty of Engineering and Natural Sciences, Sabanci University, Turkey}
\email{\href{mailto:nilay.duruk@sabanciuniv.edu}{nilay.duruk@sabanciuniv.edu}}

\author[A.~Geyer]{Anna Geyer}
\address{Delft Institute of Applied Mathematics, Faculty Electrical Engineering, Mathematics and Computer Science, Delft University of Technology, The Netherlands.}
\email{\href{mailto:a.geyer@tudelft.nl}{a.geyer@tudelft.nl}}

\author[R.~Quirchmayr]{Ronald Quirchmayr}
\address{Department of Mathematics, KTH Royal Institute of Technology, Sweden.}
\email{\href{mailto:ronaldq@kth.se}{ronaldq@kth.se}}

\begin{abstract}
We present a comprehensive introduction and overview of a recently derived model equation for waves of large amplitude in the context of shallow water waves and provide a literature review of all the available studies on this equation.  
Furthermore, we establish a novel result concerning the local well-posedness of the corresponding Cauchy problem for space-periodic solutions with initial data from the Sobolev space $H^s$ on the circle for $s>3/2$.

\vspace{1em}
\end{abstract}

\maketitle

\section{Introduction and overview}
\noindent 
This paper addresses the local well-posedness of the model equation
\begin{align}
\begin{aligned} \label{maineq}
u_{t} &+ u_x  + \frac{3 \eps}{2} u u_{x} - \frac{\delta^2}{18} (4 u_{x x x} 
+ 7   u_{x x t}) \\
&= \frac{ \eps \delta^2 }{6}\Big( 2 u_{x} u_{x x} +  u u_{x x x} \Big) 
- \frac{\eps^2 \delta^2}{96} \Big( 398 u u_{x} u_{x x} + 45 u^2 u_{x x x} + 154 u^3_{x}  \Big)
\end{aligned}
\end{align}
considered on the unit circle with initial data from the Sobolev space $H^s$ for $s>3/2$.   
Equation \eqref{maineq} was derived in~\cite{Quirchmayr2016} from the classical water wave problem for free surface gravity water waves over a flat bottom, where the underlying incompressible flow is governed by Euler's equation. It serves as an asymptotic model for the horizontal velocity component of a unidirectional shallow water wave of large amplitude at a specific depth. 

Scalar model equations are effective tools for studying the propagation of unidirectional water waves. In contrast to general solutions for the full governing equations, solutions of model equations are analytically (and numerically) tractable and provide accurate approximations for certain classes of waves, enabling an in-depth analysis of complex nonlinear wave phenomena.

One of the most prominent equations in the class of \emph{shallow water wave models} is the Korteweg-de~Vries equation (KdV) 
\bew
u_t + u_x + \frac{2 \eps}{3}u u_x + \frac{\delta^2}{6}u_{x x x} = 0;
\eew
here $\eps$ and $\delta$ are the dimensionless \emph{amplitude} and \emph{shallowness} parameter, see \cite{CJ1, Johnson, Johns1}.
KdV was first considered by Boussinesq~\cite{Bouss1877} in 1877, before Korteweg and de~Vries~\cite{KdV1895} derived it in 1895.
It exhibits a balance between nonlinear and dispersive effects enabling the formation of solitary waves, which where first observed in a narrow canal by Russell~\cite{Russell1839}. Such waves preserve their shape while propagating at constant speed and can not be described by the linear dispersive water wave equation (formally obtained by setting $\eps=0$ in the above equation), whose solutions decrease in height and disperse as they propagate.
KdV is an accurate approximation\footnote{We refer to~\cite{Lannes} for a comprehensive approach on the justification of KdV as an accurate shallow water approximation for the water wave problem.} for \emph{shallow water waves of small amplitude}, i.e.~$\delta\ll1$ and $\eps=\mathcal O(\delta^2)$, which propagate in one direction. It forms a completely integrable bi-Hamiltonian system, and the existence of a Lax pair enables the application of nonlinear Fourier transform techniques to solve the Cauchy problem exactly for large classes of initial data and to determine the long-time behavior of the solution. Moreover, there exists a single global canonical coordinate system for space-periodic solutions (so-called Birkhoff coordinates) in which the time evolution becomes linear \cite{KappelerPoeschel03}.
KdV is the simplest water wave model whose solitary waves are solitons (see the discussion in \cite{DJ}), i.e.~these localized waves preserve their shape after interacting with other waves of this type in conjunction with a phase shift. 
KdV has global solutions  in time for a very large class of initial data, cf.~\cite{T}, covering all physically relevant initial states (see the discussion in \cite{Constantin}). In particular KdV is not capable of describing \emph{breaking waves}---waves that remain bounded whereas their slope becomes infinite in finite time.
KdV belongs to a wider class of weakly nonlinear dispersive equations, the BBM equations~\cite{BBM}
\bew
u_t + u_x + \frac{2 \eps}{3}u u_x 
+ \delta^2 ( \alpha u_{x x x} + \beta  u_{x x t}  ) = 0,
\eew
with $\beta \leq 0$ (to avoid linear ill-posedness) and $\alpha = \frac{1}{6} +\beta$. Since this one-parameter family of equations does not account correctly for waves whose behavior is more nonlinear than dispersive, one is forced to consider more general regimes to include breaking waves, see~\cite{Constantin}.

A natural regime for describing larger waves is the shallow water regime for waves of \emph{moderate amplitude}, i.e.~$\delta \ll 1$ and $\eps=\mathcal O(\delta)$. 
It contains the following two-parameter family of equations, which accurately approximate the water wave problem (cf.~\cite{CL,Lannes}):
\bew 
u_t + u_x + \frac{3\eps}{2} u u_x + \delta^2 (\alpha u_{x x x} + \beta u_{x x t}) = \eps \delta^2 (\gamma u u_{x x x}  + \zeta u_x u_{x x}).
\eew
The function $u=u|_{z_0}$ describes the evolution of the horizontal velocity component of the flow field at a fixed dimensionless depth  $z_0 \in [0,1]$; the second parameter $p\in \mathbb R$ can be chosen such that the coefficients $\alpha, \beta, \gamma, \zeta$ satisfy 
\bew
\alpha = p+\lambda, \quad \beta = p - \frac{1}{6} + \lambda, \quad \gamma = -\frac{3}{2} p - \frac{1}{6} - \frac {3}{2} \lambda,
\quad \zeta = - \frac{9}{2} p - \frac{23}{24} - \frac{3}{2} \lambda,
\eew
where $\lambda = \frac{1}{2}(z^2_0 - \frac{1}{3})$.
It turns out that this family contains models that capture wave breaking. The most prominent examples are the Camassa-Holm equation (CH) 
\bew 
U_t + \kappa U_x + 3 U U_x - U_{x x t} = 2 U_x U_{x x} + U U_{x x x},  
\eew
and the Degasperis-Procesi equation (DP)
\bew 
U_t + \hat \kappa U_x + 4 U U_x - U_{x x t} = 3 U_x U_{x x} + U U_{x x x}.
\eew
By means of a suitable scaling of $U$ (provided that $\kappa,\hat\kappa\neq0$) one can recover that both CH and DP belong to the before mentioned two-parameter family, cf.~\cite{Johns1, CL}; they generalize KdV/BBM in the sense that they degenerate to KdV/BBM when $\eps=\mathcal O(\delta^2)$. 
CH was derived in~\cite{CH} as an integrable shallow water wave model with peaked solitons; it was earlier derived in \cite{FF} as a novel bi-Hamiltonian equation. DP was first considered in \cite{DP}. 
CH and DP share with KdV the rare property that they form completely integrable bi-Hamiltonian systems---in fact they are the only integrable equations of the before mentioned family, cf.~\cite{I1,I2}. We refer to \cite{CMcK} and \cite{CIL} for rigorous approaches on integrability of CH and DP, respectively. The classes of solutions of CH and DP are richer than the waves described by KdV: apart from capturing the phenomenon of wave breaking \cite{CE1, CE2, CE3, ELY}, there is another fact that distinguishes CH and DP from KdV: they possess singular traveling wave solutions such as periodic and solitary waves with peaked and cusped crests, and their solitons are peaked \cite{Lenells2005a, Lenells2005b, CE1, DHH}.

The model equation for the free surface $\eta$ of moderate amplitude shallow water waves that corresponds to CH and DP---we denote it by SE---is given by
\begin{align*} 
\eta_t + \eta_x + \frac{3\eps}{2} \eta \eta_x  -\frac{3\eps^2}{8} &\eta^2 \eta_x 
+ \frac{\delta^2}{12} \eta_{x x x} - \frac{\delta^2}{12} \eta_{x x t} \\
&= - \frac{3\eps^3}{16} \eta^3 \eta_x   - \frac{7\eps \delta^2}{24} \Big( 2 \eta_x \eta_{x x} + \eta \eta_{x x x} \Big).
\end{align*}
It initially appeared in \cite{Johns1} in Johnson's re-derivation of CH directly from the governing equations, and was later studied in various papers. 
We refer to \cite{CL} for a rigorous justification of SE as an accurate shallow water approximation for the water wave problem (besides CH and DP), a local well-posedness result for initial data in $H^s$, $s>5/2$, as well as wave breaking results. 
An improvement of local well-posedness in terms of lower Sobolev exponents can be found in \cite{Mut2}; we refer to \cite{Mut} for an analogue well-posedness result on the periodic domain and wave breaking results for periodic solutions. 
Moreover, it was proved in \cite{DGM} that the data-to-solution map is not uniformly continuous.  
SE has smooth solitary traveling wave solutions \cite{AG}, which are orbitally stable \cite{DG}, as well as smooth periodic waves and a rich collection of singular traveling waves \cite{GG}. Moreover, symmetric solutions of this equations must be traveling waves \cite{AG2}, adding evidence in support of the strong correlation between symmetry and traveling waves often observed in water wave theory. 

Equation \eqref{maineq} was derived to account for even larger waves than those captured by CH and DP (and SE) by appealing to the approach presented in \cite{Johns1}, where Johnson re-derived CH directly from the governing equations for gravity water waves by means of double asymptotic expansions in the wave parameters $\delta$ and $\eps$ of the velocity field, the pressure and the free surface. This ansatz makes it possible to treat the amplitude and shallowness parameters independently over the course of the derivation. After determining sufficiently many terms in the asymptotic expansion, one may freely choose a specific wave regime to discover the corresponding asymptotic models. 
Equation \eqref{maineq} lives in the regime which is determined by the relation $\eps = \mathcal O(\sqrt{\delta})$ and $\delta \ll 1$. 
This means that the amplitude parameter can---unlike in the CH regime---exceed the size of the shallowness parameter, which brings in the additional highly nonlinear $\eps^2\delta^2$-terms when compared to CH and DP; for this reason we speak about a shallow water equation for waves of \emph{large amplitude}. 
Equation \eqref{maineq} belongs to a two parameter family of shallow water equations of large amplitude: one of the free parameters  determines the fixed depth of the water, whereas the second parameter stems from the 
``BBM-trick'' to include the $u_{x x t}$-term in \eqref{maineq}. 
The latter parameter was chosen such that the coefficients of the quadratic third order terms in \eqref{maineq} satisfy the ratio $2:1$ as in the case for CH, 
whereas the particular choice for the dimensionless water depth entails that the $\eps^2\delta^2$-term can be written as an $x$-derivative,  which structurally distinguishes \eqref{maineq} from all the other equations from the same family.

Equation \eqref{maineq}, as well as the entire two-parameter family describing the horizontal velocity component at any depth associated with \eqref{maineq}, and the corresponding one-parameter family of equations for the evolution of the free surface elevation are locally well-posed on the real line in the classical $H^3(\R)$-setting. This was shown in \cite{Quirchmayr2016} with the help of Kato's semigroup approach for quasilinear evolution equations (see \cite{Kato}), where also a blow-up criterion for solutions $u$ of \eqref{maineq} in $H^3(\R)$  was established by means of the blow-up of $\|u_x\|_{L^\infty}$. Since the appearance of this initial study, several papers by different authors addressed the local well-posedness of \eqref{maineq} in more general settings, refined estimates for blow-up, proved persistence of certain asymptotic properties of the solution; furthermore weak traveling waves were fully classified. We sketch this progress in the following.

In \cite{YangXu}, the authors exploited the aforementioned structural properties of \eqref{maineq} to reformulate it as a nonlocal first order equation to show local well-posedness of the corresponding Cauchy problem with initial data in $H^s(\R)$ for $s>3/2$ (where $3/2$ is the sharp lower bound) by applying Kato's theory. Furthermore, they proved an asymptotic persistence property for $H^s(\R)$ solutions---certain decay rates of the initial wave profile are preserved under the time evolution. Even though the proofs leave much to be desired in terms of mathematical rigor and completeness of the arguments, see Remark \ref{Rem_final}, these results can be considered to be correct. 

The paper~\cite{GQ18} classifies all (weak) traveling wave solutions of \eqref{maineq} in $H^1_{\mathrm{l o c}}(\R)$ by means of a planar dynamical system approach. In particular it is shown that, in addition to smooth and singular traveling waves  which are similar to those of CH and DP, equation \eqref{maineq} exhibits compactly supported traveling waves of depression with cusped or crested troughs and very peculiar waves having e.g.~both peaked crests and troughs; such wave forms do not appear among the traveling waves of CH or DP. 

In \cite{FanYan} the local well-posedness for solutions of \eqref{maineq} on the real line was improved to the Besov space setting $B^{3/2}_{2,1}(\R)$. The authors employed the earlier mentioned nonlocal form of \eqref{maineq}. In contrast to the previous well-posedness results they based their proof on an iterative scheme, which was implemented by virtue of Littlewood-Paley decompositions.   
Moreover, they showed that if a solution $u=u(x,t)$ blows up in $B^{3/2}_{2,1}(\R)$ with maximal life span $T$ then 
$\int^T_0 \| u_x(t) \|^2_{L^\infty(\R)} \, \mathrm d t=\infty$. Additionally they proved that certain real analytic initial data $u_0=u_0(x)$ pass on analyticity to the solution $u$ (with respect to both the space and time variable) on some local time interval. 

The paper \cite{Zhou} considers local well-posedness of \eqref{maineq} in the Besov space $B^s_{p,q} (\R)$, $1\leq p,q \leq \infty$ and $s > \max\{1+1/p,3/2\}$. In addition, the author presents a blow-up criterion for solutions in $H^s(\R)$, $s>3/2$, in terms of the blow up of $\|u_x(t)\|_{L^\infty(\R)}$ as $t$ approaches the maximal time of existence. Furthermore, asymptotic persistence estimates of solutions are established in $L^p_\phi$ spaces for certain weight functions $\phi$. 

\smallskip
The goal of this paper is to address the space-periodic Cauchy problem for \eqref{maineq}. Our main result establishes  local well-posedness for $1$-periodic solutions in  $H^s(\T)$, $s> 3/2$, cf.~Theorem \ref{wpresult}. This constitutes a new important aspect concerning the initial value problem of \eqref{maineq} which has not been considered in the literature so far: the local well-posedness results in \cite{FanYan, Quirchmayr2016, YangXu, Zhou} have in common that they are formulated for decaying solutions on the real line and hence do not include periodic solutions.
In particular the smooth periodic traveling wave solutions of \eqref{maineq}, which were identified in \cite{GQ18}, are captured by Theorem \ref{wpresult}. Hence this result provides a rigorous foundation for the study of perturbations of these traveling waves opening the path for a stability analysis of smooth periodic traveling waves of \eqref{maineq}. 
The proof of Theorem \ref{wpresult} relies---in contrast to the well-posedness results in~\cite{Quirchmayr2016}---on a reformulation of \eqref{maineq} as a nonlocal first order equation (like in \cite{FanYan, YangXu, Zhou}), which is possible due to the particular structure of the $\eps^2\delta^2$-coefficients; it makes use of Kato's semigroup approach.


\section{Local well-posedness result} \label{sec2}
\noindent 
We address the local well-posedness of the Cauchy problem for \eqref{maineq} on the unit circle. 
That is, for a given initial data $u_0\colon \R \to \R$ with $u_0(x)=u_0(x+1)$ for all $x\in\R$ we seek a unique solution $u$, which satisfies   
\be \label{CPmaineq}
\begin{cases} 
(\ref{maineq}) \; \text{is satisfied for} \qquad  &x\in \mathbb R, \, t>0 \\
u(0,x) = u_0(x)  &x\in \mathbb R \\
u(t,x) = u(t,x+1) &x\in \mathbb R, \, t>0
\end{cases}
\ee
and depends continuously on the initial data $u_0$.
By applying Kato's semigroup approach for quasilinear evolution equations~\cite{Kato} we show that \eqref{CPmaineq} is locally well-posed in the Sobolev space $H^s_\p=H^s(\T)$ for $s>3/2$ with the usual norm $\|\cdot\|_s$; $\T$ denotes the unit circle, i.e.~$\T=\R/\Z$. In particular, $L^2_\p = H^0_\p$, the $L^2$-norm coincides with $\|\cdot\|_0$ which is induced from the standard inner product $(\cdot,\cdot)_0$ in $L^2_\p$. In full detail our main result reads as follows.
\begin{theorem}\label{wpresult}
Let $s> 3/2$. For every $u_0 \in H^s_\p$ there is a maximal time of existence $T\in (0,\infty]$ and a unique solution $u \in \mathcal C ([0,T);H^s_\p) \cap \mathcal C^1 ((0,T);H^{s-1}_\p)$ of the Cauchy problem \eqref{CPmaineq}.

Moreover, the mapping which maps initial data $u_0$ to unique local solutions $u$ of \eqref{CPmaineq} is continuous as a mapping from a neighborhood of $u_0$ in  $H^s_\p$ to $\mathcal C ([0,\tilde T];H^s_\p)$, where $[0,\tilde T]$ denotes a sufficiently small but positive closed time interval.
\end{theorem}


\section{Proof of Theorem~\ref{wpresult}}
In the following we provide a suitable adaptation of Kato's general theory \cite{Kato} for our particular situation, which serves as the basis of our proof. 
Let us consider the abstract quasilinear initial value problem in a Hilbert space $X$: 
\be \label{abstractivp}
\begin{cases} 
y_t = A(y) y + f(y) \qquad t>0\\
 y(0) = y_0. 
\end{cases}
\ee
Let $Y$ be another Hilbert space which is continuously and densely embedded into $X$, and let $S\colon Y \to X$ be a topological isomorphism. We assume the following:
\begin{enumerate}[(I)]
\item \label{K1}
For any given $r>0$ it holds that for all $y \in  \mathrm B_r(0) \subseteq Y$ (the ball around the origin in $Y$ with radius $r$) the linear operator $A(y)\colon X \to X$ generates a strongly continuous semigroup $T_y(t)$ in $X$ which satisfies
\bew
\| T_y(t) \|_{\mathcal L(X)} \leq \mathrm e^{\omega_r t} \quad \text{for all} \quad t\in [0,\infty)
\eew
for a uniform constant $\omega_r > 0$.
\item \label{K2}
$A$ maps $Y$ into $\mathcal L(Y,X)$; more precisely the domain $D(A(y))$ contains $Y$, and the restriction $A(y)|_Y$ belongs to $\mathcal L(Y,X)$ for any $u\in Y$. Furthermore, $A$ is Lipschitz continuous in the sense that for all $r>0$ there exists a constant $C_1$ which only depends on $r$ such that
\bew
\| A(y) - A(z) \|_{\mathcal L(Y,X)} \leq C_1 \, \|y-z\|_X 
\eew
for all $y,z$ inside $\mathrm B_r(0) \subseteq Y$.
\item \label{K3}
For any $y\in Y$ there exists a bounded linear operator $B(y) \in \mathcal L(X)$ satisfying $B(y) = S A(y) S^{-1} - A(y)$, and $B \colon Y \to \mathcal L(X)$ is uniformly bounded on bounded sets in $Y$. Furthermore, for all $r>0$ there exists a constant $C_2$ which depends only on $r$ such that  
\bew
\| B(y) - B(z)\|_{\mathcal L(X)} \leq C_2 \, \|y-z\|_Y
\eew
for all $y,z \in \mathrm B_r(0)\subseteq Y$.
\item \label{K4}
The map $f\colon Y \to Y$ is locally $X$-Lipschitz continuous in the sense that for every $r>0$ there exists a constant $C_3>0$, depending only on $r$, such that
\bew
\| f(y) - f(z)\|_{X} \leq C_3 \, \|y - z\|_X \quad \text{for all} \; y,z \in \mathrm B_r(0) \subseteq Y,
\eew 
and locally $Y$-Lipschitz continuous in the sense that for every $r>0$ there exists a constant $C_4>0$, depending only on $r$, such that
\bew
\| f(y) - f(z)\|_{Y} \leq C_4 \, \|y - z\|_Y \quad \text{for all} \; y,z \in \mathrm B_r(0) \subseteq Y.
\eew 
\end{enumerate}

\begin{theorem}[Kato~\cite{Kato}] \label{katothm}
Let the assumptions \eqref{K1}--\eqref{K4} hold true. Then for any $y_0 \in Y$ there exists a maximal time of existence $T\in (0,\infty]$ and a unique solution $y \in \mathcal C ([0,T);Y) \cap \mathcal C^1 ((0,T);X)$ of the abstract Cauchy problem \eqref{abstractivp} in $X$.

Moreover, the solution depends continuously on the initial data, i.e.~the mapping which maps initial data $y_0$ to unique local solutions $y$ of \eqref{abstractivp} is continuous from a neighborhood of $y_0$ in $Y$ to $\mathcal C ([0,\tilde T];Y)$ for sufficiently small $0<\tilde T \leq T$.
\end{theorem}

The first step of our proof provides a reformulation of equation \eqref{maineq} into the form \eqref{abstractivp}. To simplify the notation, let us consider the quasilinear evolution equation
\begin{align}
\begin{aligned} \label{gen_eq}
u_{t}  - \mu u_{x x t} &= \alpha_1 u_x  + \alpha_2  u_{x x x} 
  + \alpha_3 u u_{x} \\
&+\beta_1  u_{x} u_{x x} + \beta_2 u u_{x x x} 
+  \gamma_1 u u_{x} u_{x x} + \gamma_2 u^2 u_{x x x} + \gamma_3  u^3_{x},
\end{aligned}
\end{align}
with $\mu>0$ and $\alpha_1, \alpha_2, \alpha_3, \beta_1, \beta_2, \gamma_1, \gamma_2, \gamma_3 \in \R$. Let $\Lambda_\mu$ denote the first order pseudo differential operator $(1- \mu \partial^2_x)^{1/2}$.
If the coefficients $\gamma_i$, $i=1,2,3$, satisfy
\be \label{rel_gamma}
\gamma_1=2(\gamma_2+\gamma_3),
\ee
an application of $\Lambda^{-2}_\mu$  to \eqref{gen_eq} yields the abstract form
\be \label{maineq_abstract}
u_t + \bar A (u) u = \bar f(u),
\ee
where the quasilinear part $\bar A(u)$ is given by
\be \label{maineq_abstract_quasipart}
\bar A(u) = \mu^{-1} (\alpha_2 + \beta_2 u + \gamma_2 u^2) \partial_x,
\ee
and the semilinear part $\bar f(u)$ satisfies
\be \label{maineq_abstract_semipart}
\bar f(u) = \Lambda^{-2}_\mu \partial_x \bigg[ 
\Big(\alpha_1 + \frac{\alpha_2}{\mu}\Big) u + \Big(\frac{\alpha_3}{2} + \frac{\beta_2}{2\mu}\Big) u^2 
+ \frac{\gamma_2}{3\mu} u^3+ \frac{\beta_1- 3 \beta_2}{2} u^2_x + (\gamma_3-2\gamma_2) u u^2_x \bigg].
\ee

\begin{remark} \label{rem_reform}
The above procedure requires $\mu\neq0$; if $\mu$ is negative, the linear part of \eqref{gen_eq} becomes ill-posed, which explains the assumption that $\mu$ is strictly positive.

If \eqref{rel_gamma} is satisfied, which is equivalent to the property that the cubic $\gamma_i$-terms can be written as an $x$-derivative, then \eqref{gen_eq} yields a scalar conservation law of the form $u_t + \Phi(u)_x = 0$. This is an immediate consequence of the reformulation \eqref{maineq_abstract} with \eqref{maineq_abstract_quasipart} and \eqref{maineq_abstract_semipart}, since the operators
$\Lambda^{-2}_\mu$ and $\partial_x$ commute. Indeed, the commutation property of these operators can be seen from their representations 
\bew
\partial_x f = \mathcal F^{-1}(-\I n \hat f_n), \quad 
\Lambda^{-2}_\mu f = \mathcal F^{-1}\big((1+\mu n^2)^{-1} \hat f_n\big),
\eew 
where $\mathcal F\colon \mathcal C^{\infty}(\T) \to \mathcal S(\Z)$ denotes the Fourier transform and $(\hat f_n)_{n\in\Z}$ is the sequence of Fourier coefficients associated with the Fourier series of $f\in \mathcal C^{\infty}(\T)$; here $\mathcal S(\Z)$ denotes the Schwartz space of all rapidly decreasing sequences. 

The reformulation of \eqref{maineq} as a scalar conservation law opens the way to study global weak solutions of class $\mathcal C([0,\infty),H^1)$ by introducing a distributional formulation of the corresponding Cauchy problem via multiplication of the scalar conservation law by suitable test functions $\phi(t,x)$ and integration over the time and space domain. This low-regularity class includes---unlike $H^s$ with $s>3/2$---functions with peaks or even cusps; we refer to \cite{GQ18} for the existence of weak traveling wave solutions for \eqref{maineq} exhibiting such singularities.
\end{remark}

Since \eqref{maineq} is of the form \eqref{gen_eq} and satisfies \eqref{rel_gamma} it admits the reformulation in \eqref{maineq_abstract}. 
We may now normalize the coefficients in this reformulation, since their particular values do not play a role in the subsequent analysis: they merely contribute to estimation constants, whose precise value is not of special interest. Therefore, we may henceforth  consider the abstract operator equation 
\be \label{eq_normalized1}
u_t +  A (u) u =  f(u),
\ee
where
\be \label{eq_normalized2}
 A(u) = a(u) \partial_x \quad \text{with} \quad
 a(u) = (1 +  u +  u^2)  
\ee
and 
\be \label{eq_normalized3}
 f(u) =  \Lambda^{-2} \partial_x \big[ 
 u +  u^2 
+  u^3+  u^2_x + u u^2_x \big] \quad \text{with} \quad 
\Lambda = (1-\partial^2_x)^{1/2}.
\ee
As underlying function spaces we choose
$X\coloneqq (H^{s-1}_\p,\|\cdot \|_{s-1})$ and $Y\coloneqq (H^{s}_\p,\|\cdot \|_s)$ with $s> 3/2$, and consider the topological isomorphism $S\coloneqq\Lambda \colon Y\to X$ between these spaces, which defines an isometry, 
i.e.~$\| \Lambda u \|_{s-1} = \|u\|_s$ for all $u\in H^{s}_\p$. 
For a given function $g\in H^r_\p$ with $r>1/2$ let us denote by $M_g$ the corresponding multiplication operator on $H^r_\p$, i.e.~$M_g \colon H^r_\p \to H^r_\p,  w\mapsto g w$. Since $H^r_\p$, $r>1/2$, is closed under multiplication, $M_g$ is continuous.

The subsequent lemmas show that the quasilinear part $A$ and the semilinear part $f$ given by \eqref{eq_normalized2} and \eqref{eq_normalized3}, respectively, fulfill the assumptions \eqref{K1}--\eqref{K4} of Theorem~\ref{katothm}.

\begin{lemma} \label{lem_SG}
For $a\in H^s_\p$ the operator 
\bew
A \coloneqq M_a \partial_x = a(x) \partial_x, \quad D(A) \coloneqq \{ w \in L^2_\p \colon a w \in H^{1}_\p \}
\eew
generates a strongly continuous semigroup $T(t)$ in $L^2_\p$ which satisfies
\be \label{lem_C_0_SG_est}
\|T(t)\|_{0} \leq \e^{\omega t} 
\quad \text{for all} \quad t\geq 0 \quad
\quad \text{with} \quad  \omega = \frac{1}{2} \sup_{x\in\T}|a_x|.
\ee
\end{lemma}
\begin{proof}
The idea of the proof is to reveal $A$ as a bounded perturbation of a skew-adjoint operator. Then Stone's theorem and a standard perturbation result yield the assertion of the lemma.

\emph{Step 1}.
We first observe that $A\colon L^2_\p \supseteq D(A) \to L^2_\p$ is indeed a well-defined linear operator, i.e.~the action of $A$ on an element $w\in D(A)$ is uniquely determined and lies in $L^2_\p$. To see this, we first consider a general element $w\in L^2_\p$. Since the usual product $H^1_\p \times H^1_\p \to  H^1_\p$ admits an extension to a continuous product 
$H^1_\p \times H^{-1}_\p \to H^{-1}_\p$, Leibniz' formula for general periodic distributions restricted to this product yields that
\bew
(a w)_x = a_x w + a w_x  \quad \text{in} \; H^{-1}_\p. 
\eew
For $w\in D(A)$ we have that both $(a w)_x$ and $a_x w$ lie in $L^2_\p$, hence 
\bew
A w = (a w)_x - a_x w = a w_x \in L^2_\p.
\eew
Obviously $H^{1}_\p$ is contained in $D(A)$, so $A$ is densely defined.

\emph{Step 2}. We show that $\mathcal C^{\infty}_\p$ is a \emph{core} for $A$ in $L^2_\p$, i.e. for every $v\in D(A)$ there exists a sequence $(v_n)_n$ in $\mathcal C^{\infty}_\p$ such that $v_n \to v$ and $A v_n \to Av$ in $L^2_\p$.
Let $v\in D(A)$. We may assume that  
$\mathrm{supp}(v)\cap [0,1]$ is contained in $(0,1)$, otherwise we can employ a suitable parametrization of $\T$ with a subordinated smooth partition of unity. In this way we can consider the restriction of $v$ to $(0,1)$ and extend it to the whole real line by setting the extension identically $0$ outside of $(0,1)$. The resulting function lies in $L^2(\R)$ and we denote it again by $v$.  
Now it suffices to prove that there exists a sequence 
$(v_n)_n$ in $\mathcal C^{\infty}_c(\R)$ with $\mathrm{supp} (v_n) \subseteq (0,1)$ such that 
\bew
v_n \to v \quad  \text{and} \quad A v_n \to Av \quad  \text{in} \quad  L^2(\R).
\eew
Let $\rho\in \mathcal C^\infty_c(\R)$ be a mollifier with $\mathrm{supp}(\rho)\subseteq [0,1]$ and $\int_{\R} \rho =1$, and consider $v_n\coloneqq \rho_n \ast  v$, where $\ast$ is the convolution on $\R$ and $\rho_n(x)\coloneqq n \rho(n x)$.
Then $v_n \in \mathcal C^\infty_c((0,1))$ for $n\geq 1$ large enough and $v_n \to v$ in $L^2(\R)$.
We have to prove that 
\be \label{pr_L-core_0}
(av_n)_x \to (av)_x  \quad \text{in} \quad  L^2(\R).
\ee 

By assumption and our considerations we have that both $(a v)_x$ and $a_x v$ in $L^2(\R)$, thus also 
$av_x\in L^2(\R)$, and hence 
\be \label{pr_L-core_1}
\rho_n *(av_x) \to av_x \quad \text{in} \quad  L^2(\R).
\ee
Since $a_x \in H^{s-1}_\p$ we may identify it as an $L^\infty(\R)$ element, thus 
the multiplication with $a_x$ constitutes a continuous operation in $L^2(\R)$ and we find that 
\be \label{pr_L-core_2}
a_x v_n \to a_x v  \quad \text{in} \quad  L^2(\R).
\ee 
In view of the identity 
\begin{align*}
    (av_n)_x-(av)_x =[a_xv_n-a_xv] +[\rho_n *(av_x)-av_x] + [a(v_n)_x - \rho_n*(av_x)],
\end{align*}
which holds in $L^2(\R)$, together with \eqref{pr_L-core_1} and \eqref{pr_L-core_2} it suffices to show that 
\be \label{pr_L-core_3}
a(v_n)_x-\rho_n*(av_x) \to 0 \quad \text{in} \quad  L^2(\R)
\ee
in order to infer \eqref{pr_L-core_0}.
To this end, we consider the linear operators  
\bew
P_n v \coloneqq a(v_n)_x-\rho_n*(av_x), \quad  n\geq 1.
\eew
If $v$ is regular enough, say $v\in H^1(\R)$, then obviously $P_n v \to 0$ in $L^2(\R)$.
In the following we show that the family $\{P_n\}_{n\geq1}$ extends to a family of uniformly bounded linear operators in $L^2(\R)$. 
A well-known corollary of the uniform boundedness principle and the density of $H^1(\R)\subseteq L^2(\R)$ then yield 
that the sequence of operators $(P_n)_{n\geq1}$ converges pointwise to the trivial operator $P\equiv 0$ in $\mathcal L(L^2(\R))$; this implies in particular that \eqref{pr_L-core_3} holds  for our given $v\in D(A)$.

We first observe that every $P_n$ extends to a continuous linear operator on $L^2(\R)$. Note that 
\bew
P_n v = a((\rho_n)_x \ast v ) - (\rho_n)_x \ast (av) + \rho_n \ast (a_x v),
\eew
and for a general $v\in L^2(\R)$ and arbitrary $x\in \R$ we obtain that
\begin{align*}
\begin{aligned}
P_n v(x) &= \int_{\R}(\rho_n)_y(y) \big(a(x) - a(x-y)\big)v(x-y) \,  \mathrm d y  + \big(\rho_n*(a_x v)\big)(x)\\
             &= n^2\int_{\R} \rho_y(ny) \big(a(x)-a(x-y)\big) v(x-y) \, \mathrm d y +\big(\rho_n*(a_x v)\big)(x).
\end{aligned}
\end{align*}
By recalling that $\mathrm{supp}(\rho)\subseteq [0,1]$, we can write 
\begin{align*}
n^2\int_{\R} \rho_y(ny) \big(a(x)&-a(x-y)\big) v(x-y) \, \mathrm d y \\
&= n\int_{0}^1 \rho_y(y) \Big[a(x)-a\Big(x-\frac{y}{n}\Big)\Big] \, v\Big(x-\frac{y}{n}\Big) \, \mathrm d y. 
\end{align*}
Thus in view of the mean value theorem, we obtain the estimate 
\begin{align*}
\bigg|n^2\int_{\R} \rho_y(ny) \big(a(x)&-a(x-y)\big) v(x-y) \, \mathrm d y \bigg| \\
        &\leq \sup_{x\in\R}|a_x| \, \int_{0}^1 |\rho_y(y)|\,|y|\, \Big|v\Big(x-\frac{y}{n}\Big)\Big|\, \mathrm d y.
\end{align*} 
Let now  $C\coloneqq \sup_{x\in\R}|a_x|^2 \, \int_{0}^1 \rho^2_y(y) y^2\, \mathrm d y$. 
Then  
the Cauchy-Schwarz inequality and Fubini's theorem  yield that
\begin{align*}
\bigg\|n^2\int_{\R} \rho_y(ny) (a(x)-a(x-y)) &v(x-y)\, \mathrm d y \bigg\|^2_{0}\\
&\leq C\int_\R  \int_{0}^1 \Big[v\Big(x-\frac{y}{n}\Big)\Big]^2 \, \mathrm d y \, \mathrm d x\\
&= C \int_{0}^1\int_\R \Big[v\Big(x-\frac{y}{n}\Big)\Big]^2 \, \mathrm d x \, \mathrm d y\\
&= C\|v\|^2_{0}.
\end{align*}
Moreover, we obtain by Young's inequality that 
\begin{align*}
    \|\rho_n*(a_x v)\|_{0}  
    \leq  \| a_x v\|_0 
    \leq  \sup_{x\in\R}(|a_x|) \, \|v\|_{0}.
\end{align*}
We conclude that  
\bew
\|P_n\|_{\mathcal L(L^2_\p)} \leq K 
\eew
holds uniformly for all $n\geq1$ with $K=(\sqrt{C}+\sup_{x\in\R}|a_x|)$.  
It follows that $\mathcal C^\infty_\p$ is a core for $A$ in $L^2_\p$.

\emph{Step 3.} We show that $A_0 \coloneqq A + \frac{1}{2} M_{a_x}$ with domain $D(A_0) \coloneqq D(A)$ is skew-adjoint in $L^2_\p$. 
Let $A^*_0$ denote the adjoint of $A_0$ with domain $D(A^*_0)$, which is well-defined since $A$ (and hence $A_0$) is densely defined.
Fix $w\in D(A^*_0)$, then $A^*_0 w \in L^2_\p$, and for $\varphi \in \mathcal C^\infty_\p \subseteq D(A_0)$ we have that 
$(A_0 \varphi , w)_{0}=(\varphi, A^*_0 w)_{0}$. 
We consider the linear functional $F_w\colon \mathcal C^\infty_\p\to \R$ given by 
\bew
F_w(\varphi) \coloneqq (A_0 \varphi, w)_0 = \int_\T \Big(a \varphi_x + \frac{1}{2} a_x \varphi \Big) w \, \mathrm d x
= (\varphi, A^*_0 w)_0.
\eew
Obviously, $F_w$ is continuous with respect to the $L^2_\p$-norm. Integration by parts yields that $F_w(\varphi)$ coincides with the action of the regular periodic distribution 
$-A_0 w = -(a w_x + \frac{1}{2} a_x w)\in L^2_\p$ 
on the test function $\varphi$. 
By density of $\mathcal C^\infty_\p$ in $L^2_\p$ we infer that $a w \in H^1_\p$ and $A^*_0 w= -A_0 w$, i.e.~$w\in D(A_0)$, which implies $A^*_0\subseteq -A_0$.  

For $v\in D(A_0)$ let $v_n\coloneqq \rho_n \ast v$. 
By integration by parts and the second step of the proof we obtain for every $w\in D(A_0)$ that
\begin{align} \label{ibp}
\begin{aligned}
(A_0 w,v)_0 &= \lim_{n\to\infty} (A_0 w,v_n)_0 
=  \lim_{n\to\infty} \int_{\T} \Big( a w_x  + \frac{1}{2} a_x w  \Big) v_n \, \mathrm d x \\
&= - \lim_{n\to\infty} \int_{\T} \Big( a (v_n)_x  + \frac{1}{2} a_x v_n  \Big) w \, \mathrm d x \\
&= - \int_{\T} \Big( a v_x  + \frac{1}{2} a_x v \Big) w \, \mathrm d x
= - (w,A_0 v)_0,
\end{aligned}
\end{align}
in particular $v\in D(A^*_0)$ and $-A_0 \subseteq A^*_0$.

\emph{Step 4.} By the previous step the operator $\I A_0$ is self-adjoint and therefore Stone's theorem implies that $A_0$ is the infinitesimal generator of a strongly continuous semigroup of contractions in $L^2_\p$. A standard perturbation argument (see e.g.~\cite[Thm.~3.1.1]{Pazy}) now yields that $A$ is the generator of a strongly continuous semigroup $T(s)$ in $L^2_\p$, which satisfies the estimate 
\eqref{lem_C_0_SG_est} by noting that $a_x$ is a bounded function due to Sobolev's embedding theorem and therefore the operator norm of $M_{a_x} \colon L^2_\p \to L^2_\p$ can be estimated by
$\sup_{x\in \T} |a_x(x)|$.
\end{proof}

Let $[\cdot,\cdot]$ denote the usual commutator of two operators. In the course of our analysis we will make use of commutator and product estimates taken from \cite[Proposition B.10.(2)]{Lannes} and \cite[Lemma A1]{Kato}, respectively. These results are stated for functions and distributions defined on the real line (or $\R^d$); the corresponding assertions on the periodic domain $\T$ hold true analogously. We summarize these estimates formulated for one-dimensional periodic functions in the following lemma.

\begin{lemma}\label{L_comm}
Let $r>1/2$.	
	\begin{enumerate}[(i)]
		\item   \label{com_est_2}
		If $-1/2 < t \leq r +1$, there exists a constant $C_{r,t}>0$ such that 
		\bew
		\big\|[\Lambda^t,M_g]h\big\|_0 \leq C_{r,t} \, \|g\|_{r+1} \, \|h\|_{t-1}
		\eew
		for all $g\in H^{r +1}_\p$ and $h\in H^{t-1}_\p$.
		\item   \label{com_est_3}
		If $-r < t\leq r$, there exists a constant $C_{r,t}>0$ such that 
		\bew
		\|fg\|_t\leq C_{r,t} \, \|f\|_r \, \|g\|_t
		\eew
		for all $f\in H^{r}_\p$ and $g\in H^{t}_\p$.
	\end{enumerate}
\end{lemma}

\begin{lemma} \label{lem_SG2}
	For $a\in H^s_\p$ the operator 
	\bew
	\tilde A \coloneqq M_a \partial_x = a(x) \partial_x, \quad D(\tilde A) \coloneqq \{ w \in H^{s-1}_\p \colon a w \in H^{s}_\p \}
	\eew
	generates a strongly continuous semigroup $\tilde T(t)$ in $H^{s-1}_\p$ which satisfies
	\bew 
	\|\tilde T(t)\|_{s-1} \leq \e^{\tilde \omega t} 
	\quad \text{for all} \quad t\geq 0 \quad
	\quad \text{with} \quad  \tilde \omega = \frac{1}{2} \|a_x\|_{s-1}.
	\eew
\end{lemma}
\begin{proof}
Recall from Lemma \ref{lem_SG} that the operator $A=M_a \partial_x$ with domain $D(A)=\{ w \in L^2_\p \colon a w \in H^{1}_\p \}$ lies in $\mathcal G(1,\omega,L^2_\p)$ for $\omega =  \sup_{x\in\T}|a_x|/2$.	

Let $A_1\coloneqq \Lambda^{s-1} A  \Lambda^{1-s}$. In the following we show that $A_1-A\in \mathcal L(L^2_\p)$.	
Since $\partial_x$ commutes with $\Lambda^{s-1}$ and $\Lambda^{1-s}$, we may write
$A_1 - A = [\Lambda^{s-1},M_a] \Lambda^{1-s} \partial_x$. 
Let $w\in L^2_\p$ be arbitrary. Since $\Lambda^{1-s} \partial_x \in \mathcal L(L^2_\p,H^{s-2}_\p)$ we can directly apply Lemma \ref{L_comm} \eqref{com_est_2} to infer that
\begin{align*}
\|(A_1-A) w\|_0 &=  \|[\Lambda^{s-1},M_a] \Lambda^{1-s} \partial_x w\|_0 \\
&\leq C \, \|a\|_s \, \| \Lambda^{1-s} \partial_x w\|_{s-2} \\
&\leq C \, \|a\|_s \, \| w\|_0,
\end{align*}
where the generic constant $C$ solely depends on $s$. 

We just showed that $A_1$ is a bounded perturbation of $A$, it therefore generates a strongly continuous semigroup (cf.~\cite[Thm.~3.1.1]{Pazy}); more specifically, $A_1\in G(1,\omega^*,L^2_\p)$ for $\omega^*=\omega + \|A_1-A\|_{\mathcal L(L^2_\p)}$.
Since $\Lambda^{s-1}\colon H^{s-1}_\p \to L^2_\p$ is an (isometric) isomorphism, it follows from \cite[Thm.~4.5.8]{Pazy} that $H^{s-1}_\p$ is $A$-admissible. That is, $H^{s-1}_\p$ is an invariant subspace of $T(t)$, $t\geq0$, the strongly continuous semigroup generated by $A$, and the restriction of $T(t)$ to $H^{s-1}_\p$ is strongly continuous with respect to $\|\cdot\|_{s-1}$. 
By \cite[Thm.~4.5.5]{Pazy}, the $A$-admissibility of $H^{s-1}_\p$ entails that $\tilde A$, the part of $A$ in $H^{s-1}_\p$, generates the strongly continuous semigroup $\tilde T(t)$, which equals the restriction of $T(t)$ to $H^{s-1}_\p$. The assertion of the Lemma follows immediately.
\end{proof}	

A straight forward application of Lemma \ref{lem_SG2} shows that the operators $A(u)$, $u\in H^s_\p$, which are defined in \eqref{eq_normalized2}, satisfy the assumption \eqref{K1} of Kato's Theorem~\ref{katothm}. The following lemma establishes assumption \eqref{K2}.

\begin{lemma} \label{lem_A1}
The operator $A$ maps $H^s_\p$ into $\mathcal L(H^s_\p,H^{s-1}_\p)$. More precisely, the domain $D(A(u))$ contains $H^s_\p$, and the restriction $A(u)|_{H^s_\p}$ belongs to $\mathcal L(H^s_\p,H^{s-1}_\p)$ for any $u\in H^s_\p$. Furthermore, $A$ is Lipschitz continuous in the sense that for all $r>0$ there exists a constant $C_1$ which only depends on $r$ such that
\be \label{A_Lipschitz}
\| A(u) - A(v) \|_{\mathcal L(H^s_\p,H^{s-1}_\p)} \leq C_1 \, \|u-v\|_{s-1} 
\ee
for all $u,v$ inside $\mathrm B_r(0) \subseteq H^s_\p$.
\end{lemma}
\begin{proof}
Recall that $H^s_\p$ is by definition contained in $D(A(u))$ for arbitrary $u\in H^s_\p$ (cf.~Lemma \ref{lem_SG2}).	
The restriction  $A(u)|_{H^s_\p}$ belongs to $\mathcal L(H^s_\p,H^{s-1}_\p)$, since $\partial_x \in \mathcal L(H^s_\p,H^{s-1}_\p)$ and $M_{a(u)}\in \mathcal L(H^{s-1}_\p)$.

To see that the required estimate is satisfied, let $u,v, w \in H_\p^s$ be arbitrary. 
\begin{align*}
\|(A(u)-A(v))w\|_{s-1} &=\|(u - v +u^2  -v^2) \partial_x w\|_{s-1}\\
& \leq C \,(\|u-v\|_{s-1}+\|u^2-v^2\|_{s-1}) \, \|\partial_x w\|_{s-1} \\
&\leq C \, (1+\|u\|_{s-1}+\|v\|_{s-1})  \| w\|_s \,  \|u-v\|_{s-1}, 
\end{align*}
where $C$ denotes a generic constant. This shows that for every $r>0$ one can always find a constant $C_1$ such that \eqref{A_Lipschitz} holds uniformly for all $u,v \in \mathrm B_r(0) \subseteq H^s_\p$.
\end{proof}

The two lemmas below verify the assumptions \eqref{K3} and  \eqref{K4}, respectively; their proofs rely on Lemma~\ref{L_comm}.

\begin{lemma}  \label{lem_A2}
For any $u\in H^s_\p$ there exists a bounded linear operator $B(u) \in \mathcal L(H^{s-1}_\p)$ satisfying $B(u) = \Lambda A(u) \Lambda^{-1} - A(u)$, and $B \colon H^s_\p \to \mathcal L(H^{s-1}_\p)$ is uniformly bounded on bounded sets in $H^s_\p$. Furthermore, for all $r>0$ there exists a constant $C_2$, which depends only on $r$, such that  
\be \label{lem_B_estimate} 
\| B(u) - B(v)\|_{\mathcal L(H^{s-1}_\p)} \leq C_2 \, \|u-v\|_{H^s_\p}
\ee
for all $u,v \in \mathrm B_r(0)\subseteq H^s_\p$.
\end{lemma}

\begin{proof}
Let $u\in H^s_\p$. Since $\partial_x$ commutes with $\Lambda$ and $\Lambda^{-1}$ we obtain that
\bew
B(u)=\Lambda M_{a(u)}\partial_x\Lambda^{-1}-M_{a(u)}\partial_x
=[\Lambda,M_{a(u)}]\Lambda^{-1}\partial_x.
\eew
Hence we can write $\Lambda^{s-1} B(u)$ as
\begin{align*}
\Lambda^{s-1} B(u) 
&=[\Lambda^{s},M_{a(u)}]\Lambda^{-1}\partial_x + [M_{a(u)},\Lambda^{s-1}]\partial_x.
\end{align*}
Let now $w\in H^{s-1}_\p$ and $u,v \in H^s_\p$ be arbitrary. In view of the above identity and Lemma \ref{L_comm} \eqref{com_est_2} we obtain the following estimate
\begin{align*}
\big\|\big(B(u)-B(v)\big)w\big\|_{s-1} &= \big\|\Lambda^{s-1}\big(B(u)-B(v)\big) w\big\|_0 \\
&\leq \big\|[\Lambda^{s},M_{a(u)-a(v)}]\Lambda^{-1}\partial_x w \big\|_0
+ \big\|[\Lambda^{s-1},M_{a(u)-a(v)}]\partial_x w \big\|_0\\
&\leq C \|a(u)-a(v)\|_s \big( \big\|\Lambda^{-1}\partial_x w\big\|_{s-1}+\big\|\partial_x w\big\|_{s-2}\big)\\	
&\leq C \big(\|u-v\|_s+\|u^2-v^2\|_s \big)\|w\|_{s-1},
\end{align*}
where $C$ is a generic constant independent of $u, v$ and $w$. In particular, this shows that $B(u)$ extends to a bounded linear operator on $H^{s-1}_\p$ for every $u\in H^s_\p$ such that $B\colon H^s_\p \to \mathcal L(H^{s-1}_\p)$ is uniformly bounded on bounded sets in $H^s_\p$. Furthermore, this estimation proves that there exists a constant $C_2$ depending only on the radius of the ball $\mathrm B_r(0) \subseteq H^s_\p$ such that \eqref{lem_B_estimate} is satisfied for all $u,v\in B_r(0)$. 
\end{proof}

\begin{lemma} \label{Lips}
The map $f\colon H^s_\p \to H^s_\p$ is locally $H^{s-1}_\p$-Lipschitz continuous in the sense that for every $r>0$ there exists a constant $C_3>0$, depending only on $r$, such that
\be \label{X-Lipschitz}
\| f(u) - f(v)\|_{s-1} \leq C_3 \, \|u - v\|_{s-1} \quad \text{for all} \; u,v \in \mathrm B_r(0) \subseteq H^s_\p,
\ee
and locally $H^s_\p$-Lipschitz continuous in the sense that for every $r>0$ there exists a constant $C_4>0$, depending only on $r$, such that
\bew 
\| f(u) - f(v)\|_s \leq C_4 \, \|u - v\|_s \quad \text{for all} \; u,v \in \mathrm B_r(0) \subseteq H^s_\p.
\eew
\end{lemma}
\begin{proof}
Recall that 
\bew
 f(u) =  \Lambda^{-2} \partial_x \big[ 
 u +  u^2 
+  u^3+  u^2_x + u u^2_x \big].
\eew
Let $u\in H^s_\p$, then $u_x$ is contained in $H^{s-1}_\p$, which is closed under multiplication since  $s>3/2$, hence  polynomials in $u$ and $u_x$ lie again in $H^{s-1}_\p$. Since $\Lambda^{-2} \partial_x$ maps $H^{r}_\p$ continuously to $H^{r+1}_\p$ for arbitrary $r\in\R$, it is clear that $f$ is well-defined and continuous as a mapping from $H^s_\p$ to $ H^s_\p$. An application of  Lemma \ref{L_comm} \eqref{com_est_3} yields the following estimation for arbitrary $u,v \in H^s_\p$: 
\begin{align*} \label{f_X_Lip_1}
\| f(u) - f(v)\|_{s-1} 
&\leq C \,   \big\| u-v + u^2-v^2 +  u^3-v^3 +  u_x^2-v_x^2 +  uu_x^2-vv_x^2 \big\|_{s-2}\\
& \leq C\, \Big( \big( 1 + \|u+v\|_{s-1}+ \|u^2+uv+v^2\|_{s-1} \big)  \| u-v\|_{s-1} \\
& \quad \qquad+ \| u_x+v_x\|_{s-1} \|u_x-v_x\|_{s-2}  +  \| u_x\|^2_{s-1}  \|u-v\|_{s-1} \\
& \quad \qquad+ \|v(u_x+v_x)\|_{s-1} \|u_x-v_x\|_{s-2} \Big), 
\end{align*}
where $C$ denotes a generic constant being independent of $u$ and $v$.
This shows that for every $r>0$ there exists a constant $C_3>0$ such that \eqref{X-Lipschitz} holds for all $u,v \in  \mathrm B_r(0)  \subseteq H^s_\p$.

Similarly, elementary estimates (which do not rely on the product inequality in Lemma \ref{L_comm}) yield that $f$ is locally $H^s_\p$-Lipschitz continuous; we omit the details.
\end{proof}

It is now an immediate consequence of Theorem \ref{katothm} in conjunction with Lemmas \ref{lem_SG2}, \ref{lem_A1}, \ref{lem_A2} and \ref{Lips}, that the periodic Cauchy problem for equation \eqref{eq_normalized1} with $A$ and $f$ given by  \eqref{eq_normalized2} and  \eqref{eq_normalized3}, respectively, is locally well-posed for initial data in $H^s_\p$, $s> 3/2$, such that the unique local solution $u$ with maximal life span $T\in(0,\infty]$ lies in $\mathcal C ([0,T);H^s_\p) \cap \mathcal C^1 ((0,T);H^{s-1}_\p)$ and depends continuously on the initial datum. As noted earlier, the particular values of the coefficients in \eqref{eq_normalized2} and  \eqref{eq_normalized3} do not play a role when it comes to \emph{local} well-posedness. Therefore the periodic Cauchy problem for equation \eqref{maineq} is locally well-posed as well (in exactly the same sense), which finishes the proof of Theorem \ref{wpresult}.

\begin{remark} \label{Rem_final}
Our proof of Theorem \ref{wpresult} can be directly adapted to the corresponding non-periodic initial value problem on the real line. This is not surprising, since Sobolev's embedding theorem, as well as product inequalities and commutator estimates possess analogous versions for the periodic and non-periodic case.
Note that we carried over the crucial step in the proof of Lemma \ref{lem_SG} to the corresponding non-periodic scenario by means of a smooth partition of unity; in particular one immediately obtains an analogous non-periodic version of this lemma with the same proof.  

One delicate aspect of the proof is---apart from the application of the nontrivial commutator estimates in Lemma \ref{L_comm}---the approximation technique in the second step in the proof of Lemma \ref{lem_SG}, which is inspired by the arguments in \cite{CE2}.
The domain $D(A_0)$ of the operator $A_0$ in the proof of Lemma \ref{lem_SG} contains $H^1_\p$ as a proper subspace, i.e.~an element $v\in D(A_0)$ does in general \emph{not} lie in $H^1_\p$. Consequently the integration by parts formula is not directly applicable to the $L^2_\p$ inner product $(A_0 w,v)_0$  in \eqref{ibp}.

As noted in the introduction, the local well-posedness problem for \eqref{maineq} has been earlier studied in \cite{YangXu} in the analogous framework on the real line, i.e.~$X=H^{s-1}(\R)$ and $Y=H^s(\R)$. The authors chose basically the same approach, i.e.~Kato's theory was applied to the nonlocal first order reformulation of \eqref{maineq}.
A notable weakness of this paper is the presentation of the proofs: e.g.~the integration by parts formula is applied under heavy abuse of notation in an unjustified way in the proof of \cite[Lemma 3.1]{YangXu}, the proof of \cite[Lemma 3.3]{YangXu} is incomplete (only the trivial local $H^s$-Lipschitz estimate is shown, the less trivial $H^{s-1}$-estimation is omitted), nontrivial commutator estimates are applied without providing any references---to mention just a few flaws. 
\end{remark}

Let us finally comment on some open questions regarding the shallow water wave equation \eqref{maineq}.
It is up to this point unknown whether or not a solution of \eqref{maineq} remains bounded in $L^\infty$ on its whole interval of existence for the full space of initial data. This would give an answer to the important question whether or not wave-breaking is the only possible blow-up scenario (i.e.~$\|u(t)\|_{L^\infty(\R)}$ remains bounded, whereas $\limsup_{t \nearrow T} \|u_x(t)\|_{L^\infty(\R)}=\infty$ whenever the maximal life span $T$ is finite). Moreover, it is an open question if blow-up---or more specifically wave-breaking---of solutions actually occurs for certain specified classes of initial data; neither it is known---apart from the traveling wave solutions established in~\cite{GQ18}---whether there exist nontrivial initial wave profiles giving rise to global solutions.  


\vspace{2em}
\noindent
\textbf{Acknowledgments}\\
\noindent
The authors acknowledge the support of the Erwin Schr\"odinger International Institute for Mathematics and Physics (ESI) during the program ``Mathematical Aspects of Physical Oceanography''.
R.~Quirchmayr acknowledges the support of the European Research Council, Consolidator Grant No.~682537.
%


\end{document}